\documentclass{amsart}
\usepackage{amsmath,amssymb}
\usepackage{xcolor}
\def\Z{\mathbb{Z}}

\def\R{\mathbb{R}}

\def\F{\mathcal{F}}
\newcommand{\C}{\mathbb{C}}

\newcommand{\N}{\mathbb{N}}

\newcommand{\res}{\hbox{\rm res}}
\newcommand{\tr}{\hbox{\rm tr}}

\newcounter{thm}
\newcounter{ex}
\newcounter{re}
\newtheorem{Theorem}[thm]{Theorem}
\newtheorem{Lemma}[thm]{Lemma}
\newtheorem{Corollary}[thm]{Corollary}
\newtheorem{Proposition}[thm]{Proposition}

\newtheorem{remark}[thm]{Remark}
\newtheorem{Definition}[thm]{Definition}

\title[Rigid body and renormalized traces]{Rigid body equations on spaces of pseudo-differential operators with 
renormalized trace}
\author[J.-P. Magnot and E.G. Reyes]{Jean-Pierre Magnot$^1$ and 
Enrique G. Reyes$^2$}

\address{\small $^1$: {LAREMA, Universit\'e d’Angers, 2 Bd Lavoisier, 
49045 Angers cedex 1, France and Lyc\'ee Jeanne d'Arc, 40 avenue de Grande Bretagne, 63000 Clermont-Ferrand, 
France}}
\email{\small magnot@math.univ-angers.fr; jean-pierr.magnot@ac-clermont.fr}
\address{\small $^2$:
	Departamento de Matem\'{a}tica y Ciencia de la Computaci\'{o}n,
	Universidad de Santiago de Chile, Casilla 307 Correo 2, Santiago,
	Chile. }\email{\small enrique.reyes@usach.cl;
	e\_g\_reyes@yahoo.ca}
\begin{document}

\begin{abstract}
We equip the regular Fr\'echet Lie group of invertible, odd-class, classical pseudodifferential
operators $Cl^{0,*}_{odd}(M,E)$ ---in which $M$ is a compact smooth manifold and $E$ a (complex) vector bundle
over $M$--- with pseudo-Riemannian metrics, and we use these metrics to introduce a large class of 
rigid body equations. We adapt to our infinite-dimensional setting  Manakov's classical observation on 
the integrability of Euler's equations for the rigid body, and we show that our equations can be written in Lax 
form (with parameter) and that they admit an infinite number of integrals of motion. We also prove the existence 
of metric connections, we show that our rigid body equations determine geodesics on $Cl^{0,*}_{odd}(M,E)$, and we 
present rigorous formulas for the corresponding curvature and sectional curvature. Our main tool is the theory of 
renormalized traces of pseudodifferential operators on compact smooth manifolds without boundary. 
\end{abstract}

\maketitle

\textit{Keywords:} rigid body equations, pseudodifferential operators, renormalized traces, integrability.

\smallskip

\smallskip

\textit{MSC(2010): 58B20, 58B25, 70E45, 37K45.} 
\section{Introduction}
{
In this paper we continue our work on Mathematical Physics themes posed on 
spaces built out of {\em non-formal} pseudodifferential operators. In \cite{MR2020} we introduced a 
Kadomtsev-Petviashvili hierarchy with the help of odd-class non-formal pseudodifferential operators, its 
importance being that our new KP hierarchy ``covers" a KP hierarchy 
posed on spaces of symbols, that is, on equivalence classes of non-formal pseudodifferential operators; here we 
consider analogues of the rigid body equation on Fr\'echet Lie groups of non-formal pseudodifferential operators. { We recall that the rigid body equation appears thus:

We fix a Lie group $G$ with Lie algebra $\mathcal{G}$. We recall that 
the functional derivative of a smooth function $f : \mathcal{G}^{\ast}
\rightarrow {\bf R}$ at $\mu \in \mathcal{G}^{\ast}$ is the unique
element $\delta f/\delta \mu$ of $\mathcal{G}$ determined by
\begin{equation}
\left< \nu \, , \frac{\delta \, f}{\delta \, \mu} \right> \; =
\left.\frac{d}{d\,\epsilon}\right|_{\epsilon = 0} f(\mu + \epsilon\,\nu) = T_{\mu}f
\cdot \nu                                                                   \label{gradiente}
\end{equation}
for all $\nu \in \mathcal{G}^{\ast} (=T_\mu \mathcal{G}^\ast)$, in which $< \, ,
\, >$ denotes a natural pairing between $\mathcal{G}$ and $\mathcal{G}^{\ast}$, and that 
(see \cite[p. 129]{BP} or \cite[Chapter 9]{Holm}) 
the Lie-Poisson bracket on the dual space $\mathcal{G}^{\ast}$ is defined as follows:
for all smooth functions $F, G : \mathcal{G}^{\ast} \rightarrow {\bf R}$ and $\mu \in \mathcal{G}^{\ast}$,
\begin{equation}
\{F \, , \, G \}(\mu) = \left< \mu \, , \left[ \frac{\delta \,
F}{\delta \, \mu} \, ,
      \, \frac{\delta \, G}{\delta \, \mu} \right] \right>  \; .             \label{lie-poisson0}
\end{equation}
%As before, $< ~ , ~ >$ denotes a natural pairing between $\mathcal{G}$ and $\mathcal{G}^{\ast}$. 
If $\mathcal{G}^{\ast}$ is equipped with the
Lie-Poisson structure (\ref{lie-poisson0}), the Hamiltonian vector
field corresponding to a function $H :\mathcal{G}^{\ast} \rightarrow {\bf
R}$ acts on smooth functions $F :\mathcal{G}^{\ast} \rightarrow {\bf R}$
as
\begin{equation} \label{vector}
X_{H}(\mu) \cdot F = \{ H , F \} (\mu) = \left< \mu \, , \, \left[ \frac{\delta
H}{\delta \mu} \, , \, \frac{\delta F}{\delta \mu} \right] \right> \; ,
\end{equation}
and it follows that the Hamiltonian equation of motion on $\mathcal{G}^\ast$ is 
\[
\frac{d\, \mu}{dt}=X_{H}(\mu)=\left< \mu\, , \left[\frac{\delta H}{\delta \mu}\, , \, \cdot\, \right]\right>\; .
\]
Now we assume that there exists a non--degenerate 
pairing $< ~ , ~ >$ between $\mathcal{G}^{\ast}$ and $\mathcal{G}$.
Then, we can write the equation of motion as an equation on $\mathcal{G}$,
\begin{equation} \label{wee}
\left< \frac{dP}{dt}\, , \, \cdot \, \right> \; = \; 
\left< P\, , \left[ \frac{\delta H}{\delta \mu} ,\, \cdot\,  \right] \right> 
\end{equation}
We call (\ref{wee}) the Euler equation {\em in weak form}. There are two paths we can take. First, if the
pairing between  $\mathcal{G}^{\ast}$ and $\mathcal{G}$ is, in addition, symmetric and infinitesimally
$Ad$-invariant, this is, it satisfies
\begin{equation}
< P \, , \, [Q \, , \, R] > \; = \; < [R \, , \, P] \, , \, Q > \; , \quad \quad P,Q,R \in \mathcal{G} \; . 
\label{ad}
\end{equation}
then we can write Equation (\ref{wee}) in Lax form, 
\begin{equation} \label{lax10}
\frac{d \, P}{d t} = \left[ P \, , \frac{\delta \, H}{\delta \, \mu} \right] \; .
\end{equation}
Following Berezin and Perelomov, see \cite{BP}, we call Equation (\ref{lax1}) the Euler equation or, the rigid
body equation posed on $\mathcal{G}$. Second, if the pairing between $\mathcal{G}^{\ast}$ and $\mathcal{G}$ is
symmetric, but not necessarily infinitesimally $Ad$-invariant, we define an adjoint map with respect to the
pairing, $ad_{\mathbb{A}}\,$, via the equation
$$
\left< [P,Q] , R \right> = - \left< Q, ad_{\mathbb{A}}(P)\cdot R \right> \; .
$$
Then, Equation (\ref{wee})  can be written as
\begin{equation} \label{lax22}
\frac{d \, P}{d t} = - ad_{\mathbb{A}}\left(\frac{\delta \, H}{\delta \, \mu}\right)\cdot P\; .
\end{equation}
This equation is an evident generalization of (\ref{lax10}); we also call it the Euler equation posed on
$\mathcal{G}$. If the pairing $< ~ , ~ >$ is symmetric and 
positive-definite, then (\ref{lax22}) determines geodesics on $G$, see for instance \cite{T}.

Now we can state our motivation for studying the rigid body equation (in their versions (\ref{wee}) and 
(\ref{lax22})) in the (non-commutative) setting of rigorous pseudo-differential operators: we are inspired by 
Arnold's seminal work \cite{fourier} and by the non-commutative version of the Korteweg-de Vries
(KdV) equation considered by Berezin and Perelomov in \cite{BP}. In both cases they consider motion on 
infinite-dimensional Lie groups and Lie algebras, including geodesic motion. Now, previous research
stemming from \cite{fourier} has concentrated on providing rigorous analytic foundations for the beautiful 
results on Riemannian geometry of diffeomorphism groups appearing in \cite{fourier}, and on the study of some very 
interesting equations posed on these infinite-dimensional groups, see for instance \cite{EM,BBCEK,C} or the 
reviews \cite{Holm,T}. 
%Since we have already considered very general groups in our study of the Kadomtsev-Petvishvili
%hierarchy, see \cite{MR2016,MR2020}, 
We  wonder if we can study (geodesic) motion on groups 
 which are natural alternatives to diffeomorphism groups in a fully rigorous way. 

Indeed, we observe herein that {\em there exist Fr\'echet Lie groups of non-formal pseudodifferential operators 
that can be equipped with (weak) pseudo-Riemannian metrics}, and that it is very feasible to investigate rigid 
body equations posed on them. Let us explain briefly our geometric setting: 
we fix a compact smooth manifold without boundary $M$ and a (complex) vector bundle over $M$, and we prove that 
the set $G=Cl^{0,*}_{odd}(M,E)$ of all zero order, invertible, and odd-class non-formal classical pseudodifferential 
operators acting on sections of $E$, can be equipped with the structure of a regular Fr\'echet Lie group. 
Once we fix $G$, we need to define a suitable linear functional that replaces standard trace, in order to consider 
formulas analogous to 
$$(A,B)_{\mathbb{A}} \mapsto tr\left(A\, \mathbb{A}(B) \right)$$ 
which (depending on the operator $\mathbb{A}$) defines a metric on spaces of matrices.  
Now, the classical trace $tr$ of (trace-class) operators on a Hilbert space is not defined on the entire Lie 
algebra of $G.$ We choose a linear extension of $tr$, called the $\zeta-$renormalized trace or weighted trace, on   
the class $Cl^{0}_{odd}(M,E)$ of odd-class non-formal classical pseudodifferential operators introduced by 
Kontsevich and Vishik \cite{KV1,KV2} and fully described in \cite{PayBook,Scott}.
Properties of $\zeta-$renormalized traces allow us to endow the group $G$
 with the structure of an infinite-dimensional (weak) pseudo-Riemannian manifold by extending the results of 
 \cite{Ma2020-4} to our setting.
}

Our research differs from previous investigations carried out in the spirit of Arnold's \cite{fourier}, 
such as \cite{BBCEK,C,Holm,T} or \cite{GPoR}, in two ways. First, our rigid body equation is a bona fide ordinary 
differential equation (similar in this respect to the classical rigid body equation), due to the way the Lie 
bracket is defined in our context; it does not reduce to a partial differential equation as it happens if we study 
Euler's equations on diffeomorphism groups, see \cite{BBCEK,C,Holm,GPoR,T}. Second, it is a {\em non-commutative} 
nonlinear equation, since it is an nonlinear equation for an unknown rigorous pseudodifferential operator. In this 
latter aspect, our Euler equation is similar in spirit to the non-commutative version of the Korteweg-de Vries 
(KdV) equation by Berezin and Perelomov, see \cite{BP}, and to the equations considered by Olver and Sokolov in 
\cite{OS}, although, in contradistinction with {\em e.g.} \cite{OS}, our study belongs to global 
analysis rather than to formal geometry, because of the presence of non-formal (Fr\'echet) Lie groups of 
non-formal pseudo-differential operators and of traces which fully extend the trace of a finite rank operator.

Interestingly, {\em our version of Euler's equation shares with the classical rigid body 
equation the important properties of admitting a parameter-depending Lax formulation and integrals 
of motion}. We check this claim by adapting Manakov's seminal observation on the integrability of Euler's 
equation appearing in \cite{Ma}, to our framework. For example, in the particular case in which the manifold $M$ 
mentioned above is simply $S^1$, and the unknown pseudodifferential operator is a function 
$X : S^1 \rightarrow \mathbb{C}$, we can check that the equation
\begin{equation} \label{eqs62}
 \frac{dX}{dt} = X (\Delta + \pi) X^* (\Delta + \pi)^{-1} - X X^* \; ,
\end{equation}
where $\pi$ is the $L^2-$ orthogonal projection on the kernel of the Laplacian, admits infinitely 
many independent integrals of motion (at least for a large class of initial conditions, see Subsection 6.2). 

\smallskip

We organize this work as follows. In Section 2 we present a %quick 
 quick survey of the properties of pseudodifferential 
operators that we use, including the Fr\'echet structure of  $Cl^{0,*}_{odd}(M,E)$. We also 
introduce the Wodzicki residue and renormalized traces, and we state the facts that make 
them interesting  objects for geometry. 
In Section 3 we show how to construct non-degenerate pairings on spaces of non-formal classical 
pseudodifferential operators using renormalized traces. In Section 4 we use these pairings to construct 
right-invariant pseudo-Riemannian metrics on  $Cl^{0,*}_{odd}(M,E)$, we introduce our rigid 
body equations using variational methods, we show that they can be written as Lax equations depending on a 
parameter, and we present integrals of motion. In Section 5 we make some remarks on the pseudo-Riemannian geometry 
of the Fr\'echet Lie group $Cl^{0,*}_{odd}(M,E)$, and we prove that, as in
more classical contexts, our rigid body equations determine geodesics. We finish in Section 6 with
two examples: in 6.1 we present equations arising from a metric depending on the heat operator,
and in 6.2 we analyse a class of equations which includes (\ref{eqs62}). In particular, by direct computations on 
renormalized traces, we show independence of integrals of motion. 
}

\section{Preliminaries}
\subsection{Preliminaries on classical pseudodifferential operators}
We introduce groups and algebras of non-formal pseudodifferential operators needed to set up our equations. 
Basic definitions are valid for real or complex finite-dimensional vector bundles $E$ 
over a compact manifold $M$ without boundary whose typical fiber is a finite-dimensional 
{real or complex} vector space $V$. 
We begin with the following definition after \cite[Section 2.1]{BGV}.

\begin{Definition} 
	The graded algebra of differential operators acting on the space of 
	smooth sections $C^\infty(M,E)$ is the algebra $DO(E)$ generated 
	by:
	
	$\bullet$ Elements of $End(E),$ the group of smooth maps $E \rightarrow 
	E$ leaving each fibre globally invariant and which restrict to linear 
	maps on each fibre. This group acts on sections of $E$ via (matrix) 
	multiplication;
	
	$\bullet$ The differentiation operators
	$$\nabla_X : g \in C^\infty(M,E) \mapsto \nabla_X g$$ where $\nabla$ 
	is a connection on $E$ and $X$ is a vector field on $M$.
\end{Definition}

Multiplication operators are operators of order $0$; differentiation 
operators and vector fields are operators of order 1. In local 
coordinates, a differential operator of order $k$ has the form
$ P(u)(x) = \sum p_{i_1 \cdots i_r}(x) \nabla_{x_{i_1}} \cdots 
\nabla_{x_{i_r}} u(x) \; , \quad r \leq k \; ,$
{ in which $u$ is a (local) section} and the coefficients $p_{i_1 \cdots i_r}$ can be matrix-valued.
The algebra $DO(M,E)$ is filtered by order: we note by $DO^k(M,E)$,$k \geq 0$, the differential operators of 
order less or equal than $k$.  

Now we embed $DO(M,E)$ into the algebra of classical 
pseudodifferential operators $Cl(M,E)$. We need to assume that the reader is familiar with the basic facts on 
pseudodifferential operators 
defined on a vector bundle $E \rightarrow M$; these facts can be found for instance 
in \cite{Gil}, in the review \cite[Section 3.3]{Pay2014}, and in the papers \cite{BK} and \cite{Wid} in which
the authors construct a global symbolic calculus for pseudodifferential operators showing, for instance, how the 
geometry of the base manifold $M$ furnishes an obstruction to generalizing 
local formulas of composition and inversion of symbols.

\vskip 6pt
\noindent
\textbf{Notations.} 
We note by  $ PDO (M,E) $ the space of pseudodifferential operators on smooth sections of $E$, see 
\cite[p. 91]{Pay2014}; by $ PDO^o (M,E)$ the space of pseudodifferential operators of order $o$; and by 
$Cl(M,E)$ the space of classical pseudodifferential operators acting on 
smooth sections of $E$, see \cite[pp. 89-91]{Pay2014}. 
We also note by $Cl^o(M,E)= PDO^o(M,E) \cap Cl(M,E)$ the space of classical 
pseudodifferential operators of order $o$, and by 
$Cl^{o,\ast}(M,E)$ the group of units of $Cl^o(M,E)$. 

\vskip 6pt

A topology on spaces of classical pseudodifferential operators has 
been described by Kontsevich and Vishik in \cite{KV1}: it is a Fr\'echet topology (and therefore
it equips $Cl(M,E)$ with a smooth structure) such that each space 
$Cl^o(M,E)$ is closed in $Cl(M,E).$ This topology is discussed in 
\cite[pp. 92-93]{Pay2014}, see also \cite{CDMP,PayBook,Scott} for other descriptions.
We use all along this work the Kontsevich-Vishik topology. 

We set
$$ PDO^{-\infty}(M,E) = \bigcap_{o \in \Z} PDO^o(M,E) \; .$$
It is well-known that $PDO^{-\infty}(M,E)$ is a two-sided ideal 
of $PDO(M,E)$, see e.g. \cite{Gil,Scott}. 
This fact allows us to define the quotients
$$\mathcal{F}PDO(M,E) = PDO(M,E) / PDO^{-\infty}(M,E)\; ,$$ 
$$\F Cl(M,E) = Cl(M,E) /  PDO^{-\infty}(M,E)\; ,$$
and
$$ \quad \F Cl^o(M,E) = Cl^o(M,E)  / PDO^{-\infty}(M,E)\; .$$

The script font $\F$ stands for {\it  formal } pseudodifferential operators. The quotient $\mathcal{F}PDO(M,E)$ 
is an algebra isomorphic to the space of formal symbols, see \cite{BK}, 
and the identification is a morphism of $\mathbb{C}$-algebras for 
the usual multiplication on formal symbols (appearing for instance in \cite[Lemma 1.2.3]{Gil} and
\cite[p. 89]{Pay2014},  and in \cite[Section 1.5.2, Equation (1.5.2.3)]{Scott} for the particular case of
classical symbols). 

%A known result on the structure of the spaces we are using is the following. 

\begin{Theorem} \label{2} 
	The groups $Cl^{0,*}(M,E)$ and $\mathcal{F}Cl^{0,*}(M,E)$, in which  
	${\mathcal F}Cl^{0,*}(M,E)$ is the group of units of the algebra 
	${\mathcal F}Cl^{0}(M,E)$, are regular Fr\'echet Lie groups equipped with smooth exponential maps. 
	Their Lie algebras are $Cl^{0}(M,E)$ and ${\mathcal F}Cl^{0}(M,E)$ respectively.
\end{Theorem}

Regularity is reviewed in \cite{MR2016,MR2020} and also in Paycha's lectures,
see \cite[p. 95]{Pay2014}. The Lie  group structure of $Cl^{0,*}(M,E)$ is discussed in 
\cite[Proposition 4]{Pay2014}. Theorem \ref{2} is essentially proven in \cite{Ma2006}: it is noted in
this reference that the results of \cite{Gl2002} imply that the 
group $Cl^{0,*}(M,E)$ (resp. $\F Cl^{0,*}(M,E)\,$) is open in 
$Cl^0(M,E)$ (resp. $\F Cl^{0}(M,E)\,$) and that therefore it is a regular 
Fr\'echet Lie group.

Now we will introduce our main classes of classical pseudodifferential operators. First of all we recall the 
following:

If $A \in Cl^o(M,E)$, the symbol $\sigma(A)(x,\xi)$ has an asymptotic expansion of the form
\begin{equation}  \label{expansion}
\sigma(A)(x,\xi) \sim \sum_{j=0}^\infty \sigma_{o-j}(A)(x,\xi)\; , \quad (x,\xi) \in T^\ast M\; ,
\end{equation}
in which each $\sigma_{o-j}(A)(x,\xi)$ satisfies the homogeneity condition
$$
\sigma_{o-j}(A)(x,t\,\xi) = t^{o-j}\sigma_{o-j}(A)(x,\xi)\; \quad \quad \mbox{ for every } t > 0\; .
$$
The function $\sigma_{o}(A)(x,\xi)$ is the principal symbol of $A$. We define

\begin{Definition} \label{d7} 
	A classical pseudodifferential operator $A$ on $E$ is called 
	\begin{itemize}
	\item {\bf odd 
	class} if and only if for all $n \in \Z$ and all $(x,\xi) \in T^*M$ we have:
	$$ \sigma_n(A) (x,-\xi) = (-1)^n  \sigma_n(A) (x,\xi)\; ,$$
	and
	\item  {\bf even 
	class} if and only if for all $n \in \Z$ and all 
	$(x,\xi) \in T^*M$ we have:
	$$ \sigma_n(A) (x,-\xi) = (-1)^{n+1}  \sigma_n(A) (x,\xi)\; .$$
	\end{itemize} 
\end{Definition}

Odd class pseudodifferential operators were introduced in \cite{KV1,KV2}; they are called 
``even-even pseudodifferential operators'' in the treatise \cite{Scott}. 
For instance, recalling Definition 1, we see that {\em all differential operators} are odd class. 

Hereafter, the subscript ${}_{odd}$ 
(resp. ${}_{even}$) attached to a given space of (formal) pseudodifferential 
operators will refer to the set of all odd (resp. even) class (formal) pseudodifferential 
operators belonging to that space.

\smallskip

We need the following result, already essentially present in \cite{KV1,Scott}:

\begin{Lemma}
	$Cl_{odd}(M,E)$ and $Cl^0_{odd}(M,E)$ are associative algebras.
\end{Lemma}
\begin{proof} 
We work locally. Let $A,B$ be two odd class pseudodifferential operators of order $m$ and $m'$ respectively; 
the homogeneous pieces of the symbols of $A,B,AB$ are related via (see \cite[Section 1.5.2, Equation (1.5.2.3)]
{Scott})
$$
\sigma_{m+m'-j}(AB)(x,\xi) = \sum_{|\mu|+k+l=j}\frac{1}{\mu!}\partial_\xi^\mu \sigma_{m-k}(A)(x,\xi)
D_x^\mu \sigma_{m'-l}(B)(x,\xi)\; ,
$$
in which $|\mu|$ is the length of the multi-index $\mu$. We have, using the first equation appearing in 
Definition \ref{d7},
$$
\partial_\xi^\mu \sigma_{m-k}(A)(x,-\xi)= (-1)^{m-k+|\mu|}\partial_\xi^\mu \sigma_{m-k}(A)(x,\xi)
$$
and
$$
D_x^\mu \sigma_{m'-l}(B)(x,-\xi)= (-1)^{m'-l}D_x^\mu\sigma_{m'-l}(B)(x,\xi)\; ,
$$
so that 
$$
\sigma_{m+m'-j}(AB)(x,-\xi) = \sum_{|\mu|+k+l=j}\frac{1}{\mu!}(-1)^{m-k+|\mu|+m'-l}\partial_\xi^\mu 
\sigma_{m-k}(A)(x,\xi) D_x^\mu \sigma_{m'-l}(B)(x,\xi)\; .
$$
Changing $+|\mu|$ for $-|\mu|$ in $(-1)^{m-k+|\mu|+m'-l}$ and using $|\mu|+k+l=j$ we obtain
\begin{eqnarray*}
\sigma_{m+m'-j}(AB)(x,-\xi) &=& (-1)^{m+m'-j}\sum_{|\mu|+k+l=j}\frac{1}{\mu!}\partial_\xi^\mu 
\sigma_{m-k}(A)(x,\xi) D_x^\mu \sigma_{m'-l}(B)(x,\xi) \\
 &=& (-1)^{m+m'-j}\sigma_{m+m'-j}(AB)(x,\xi)\; ,
\end{eqnarray*}
	which proves the first claim. 
	That $Cl^0_{odd}(M,E)$ is an associative algebra now follows from the standard fact that zero-order classical 
	pseudodifferential  operators form an algebra, see for instance the proof of Proposition 3 in \cite{Pay2014}.	
\end{proof}

\smallskip

The next proposition singles out an interesting Lie group included in $Cl_{odd}(M,E)$.

\begin{Proposition}
	The algebra $Cl_{odd}^0(M,E)$ is a closed subalgebra of 
	$Cl^0(M,E)$. Moreover, $Cl_{odd}^{0,*}(M,E)$ is 
	\begin{itemize}
		\item an open subset of $Cl^0_{odd}(M,E)$ and,
		\item a regular Fr\'echet Lie group {with Lie algebra $Cl_{odd}^{0}(M,E)$ and smooth Lie 
		bracket $[A,B] = AB - BA$}.
	\end{itemize} 
\end{Proposition}
\begin{proof}
	We denote by $\sigma(A)(x,\xi)$ the total formal symbol of 
	$A \in Cl^0(M,E).$ We define the function
	$$
	\phi: Cl^0(M,E)\rightarrow \mathcal{F}Cl^{0}(M,E)
	$$ as
	$$\phi(A) = 
	\sum_{n \in \mathbb{N}}\sigma_{-n}(x,\xi) - (-1)^n\sigma_{-n}(x,-\xi)\; .
	$$
	This map is smooth, and $$Cl^{0}_{odd}(M,E)= Ker(\phi),$$ which 
	shows that 	$Cl_{odd}^0(M,E)$ is a closed subalgebra of   
	$Cl^0(M,E).$ 
	Moreover, if $H = L^2(M,E),$ 
	$$Cl^{0,*}_{odd}(M,E) = Cl^{0}_{odd}(M,E)\cap GL(H),$$
	which proves that $Cl_{odd}^{0,*}(M,E)$ is open in the 
	Fr\'echet algebra $Cl^0_{odd}(M,E),$ and it follows 
	that it is a regular Fr\'echet Lie group by arguing along 
	the lines of \cite{Gl2002,Neeb2007}.
\end{proof}

We finish our preliminaries on pseudodifferential operators noting that at a formal level we have the
splitting 
$$ \mathcal{F}Cl(M,E) = \mathcal{F}Cl_{odd}(M,E) \oplus \mathcal{F}Cl_{even}(M,E)\; ,$$
and the following composition rules for formal pseudodifferential operators $A\circ B:$ 
\vskip 12pt
\begin{tabular}{|c|c|c|}
	\hline
	&&\\
	& $A$ odd class & $A$ even class \\
	&&\\
	\hline
	&&\\
	$B$ odd class & $A \circ B$ odd class & $A \circ B$ even class \\
	&&\\
	\hline
	&&\\
	$B$ even class & $A \circ B$ even class & $A \circ B$ odd class \\
	&&\\
	\hline
\end{tabular}

\smallskip

\subsection{Renormalized traces} \label{s3}

Hereafter we assume that the  typical fiber of the bundle $E$ is a complex vector space, and that 
$E$ is equipped with an Hermitian product $<\cdot ,\cdot >\,$. An excellent review of this geometric set-up 
appears in \cite[Chapters III, IV]{Wells}. This product
allows us to define the following $L^2$-inner product on sections of $E$: 
$$ \forall u,v \in C^\infty( M,E), \quad (u,v)_{L^2} = \int_{M} <u(x),v(x)> dx\; , $$
where $dx$ is a fixed Riemannian volume on $M$. 
%determined by the Riemannian product on $E$ induced by $<\cdot ,\cdot >\,$. 

We need to use some further notions of the theory of pseudodifferential operators. First of all, we use the inner 
product just introduced to define {\em self-adjoint} and {\em positive} pseudodifferential operators. We also
define {\em elliptic} pseudodifferential operators: a classical pseudodifferential operator $P$ of order $o$ is
elliptic if its main symbol $\sigma_o(P)(x,\xi):E_x \rightarrow E_x$ is invertible, see for instance 
\cite[Chapter IV, Section 4]{Wells} or \cite[p. 92]{Scott}; these pseudodifferential operators are also discussed
quickly in \cite[Definitions 6.17, 6.31]{PayBook}. We denote by $Ell(M,E)$ the space of all classical elliptic 
pseudodifferential operators. 

\begin{Definition}
	$Q$ is a \textbf{weight} of order $q \in \N^*$ on $E$ if and only if $Q$ is a classical, elliptic,
	self-adjoint and positive pseudodifferential operator acting on smooth sections of $E$.
\end{Definition}

Under these assumptions, the weight $Q$ has a real discrete spectrum, and 
all its eigenspaces are finite dimensional. Moreover, for such a weight $Q$ of order $q$, we can define complex 
powers of $Q$, see e.g. \cite{CDMP} or \cite[Section 7.1]{PayBook} for a quick overview of technicalities and 
further references: the powers $Q^{-s}$ of the weight $Q$ are defined for $Re(s) > 0$ using a contour integral of 
the form 
$$ Q^{-s} = \int_\Gamma \lambda^s(Q- \lambda Id)^{-1} d\lambda\; ,$$
in which $\Gamma$ is a contour around the real positive axis {that appears precisely identified in 
\cite[Section 7.1]{PayBook}. The pseudodifferential operator $Q^{-s}$ is a classical pseudodifferential
operator of order $- q\,s$. }

\smallskip

Now we let $A$ be a log-polyhomogeneous pseudodifferential operator, that is, $A$ is a pseudodifferential 
operator such that its symbol is, locally, of the form
$$
\sigma(x,\xi) \sim \sum_{j=0}^o \,\sum_{\,-\infty <k}^{o'} \sigma_{j,k}(A)(x,\xi)\,\log(|\xi|)^{j} \; ,
$$  
in which $\sigma_{j,k}(A)(x,\xi)$ are classical symbols, see \cite[Section 2.6]{Scott}. Within this general 
framework we introduce zeta functions and traces. The map
$$
\zeta(A,Q,s) = s\in \mathbb{C} \mapsto \hbox{tr} \left( AQ^{-s} \right)\in \mathbb{C}\; ,
$$ 
in which $\tr$ is the classical trace of trace-class pseudo-differential operators, see 
\cite[Section 1.3.5.1]{Scott}, is well-defined for $Re(s)$ large enough, and it extends to a meromorphic function 
on $\C$ with possibly a pole at $s = 0$ \cite{PayBook,Scott}. When $A$ is classical, this pole is a simple pole, 
and when $A$ is classical, odd-class, and $M$ is odd dimensional, $\zeta(A,Q,s)$ has no pole at $s=0.$ 

Gilkey \cite[Section 1.12.2]{Gil} treats zeta functions and their relation to the heat kernel in detail; Scott 
\cite[Section 1.5.7]{Scott} deals with zeta functions in a very general setting: he extends the computations 
of Kontsevich and Vishik \cite{KV1}.

When $A$ is a classical pseudodifferential operator, the Wodzicki residue, $res_W$, appearing in \cite{W}, see 
also \cite{Ka}, is directly linked with the simple pole of  $\zeta(A,Q,.)$  at $0$  
by the residue formula 
\begin{equation} \label{alpha}
res_{s = 0}\zeta(A,Q,s) = (1/ q)\, \res_W A \; .
\end{equation}

The Wodzicki residue is a higher dimensional analog of the Adler trace on formal symbols introduced in 
\cite{Adl}, but we remark that the former fails to be a direct extension of the latter. For example, 
$(1 + \frac{d}{dx})$ is invertible in $ Cl (S^1,\C)$, and since $(1 + \frac{d}{dx})^{-1}$ is odd class, 
we have on one hand that $$res_W(1 + \frac{d}{dx})^{-1} = 0 \; ,$$
see {\em e.g.} \cite[Section 1.5.8.2]{Scott}, but on the other hand, the formal symbol of 
$(1 + \frac{d}{dx})^{-1}$ has a non-vanishing Adler trace.

Following \cite[Chapter 7]{PayBook} and \cite[Section 1.5.7]{Scott}, see also \cite{CDMP}, we define renormalized 
traces of classical pseudodifferential operators as follows:

\begin{Definition} \label{d6} 
Let $A$ be a log-polyhomogeneous pseudo-differential operator {and $Q$ a fixed weight of order $q$.} 
	The finite part of $\zeta(A,Q,s)$ at $s = 0$ is called the renormalized trace of $A$. We denote it by 
	$\tr^Q A$. If $A$ is a classical pseudodifferential operator, then
$${ \tr^Q A = lim_{s \rightarrow 0} \left(\hbox{tr} (AQ^{-s}) - \frac{1}{q\,s} \res_W (A)\right) .}$$
\end{Definition}

If $A$ is a trace-class { pseudodifferential operator} acting on  $L^2(M,E)$ then 
$\hbox{tr}^Q{(A)}=\hbox{tr}{(A)}$, see {\em e.g.} \cite{CDMP}. However, generally speaking, 
the linear functional $\hbox{tr}^Q$ {\em is not} a trace, this is, it does not vanish on commutators, 
although the linear map $\res_W$ determined by the Wodzicki residue does fulfil the trace property. 

\smallskip

We state the main properties of $res_W$ and of renormalized trace in Propositions \ref{p5} and \ref{p6}.

\begin{Proposition} \label{p5}
	\begin{item}
		(i) The Wodzicki residue $\res_W$ is a trace on the algebra of
		classical pseudodifferential operators $Cl(M,E)$, i.e. { $\forall
		A,B \in Cl(M,E), \res_W[A,B]=0.$ }
	\end{item}
	\begin{item}
		(ii) if $m =
		dim M$ and $A \in Cl(M,E)$,
		$$ { \res_W A }= \frac{1}{(2\pi)^n} \int_M \int_{|\xi|=1} tr \sigma_{-m}(x,\xi) d\xi dx $$
		where $\sigma_{-m}$ is the $(-m)$ positively homogeneous part of the
		symbol of A, see $(\ref{expansion})$. In particular, $\res_W$ does not depend on the choice
		of $Q$, in spite of what $(\ref{alpha})$ may suggest.
	\end{item}
	
	%\begin{item}
	%	(iii) Let $\lambda:Cl(S^1,E) \rightarrow \C$ be trace. If $ dim M
	%	\geq 2$, then $ \exists k \in \C, \lambda=k \res.$
	%\end{item}
\end{Proposition}

\begin{Proposition} \label{p6}
Let us fix a weight $Q$ of order $q$.   

	\begin{itemize}
		\item   Given two classical pseudo-differential operators A and B,
	%	given a weight Q,
		\begin{equation}\label{crochet}  
		\tr^Q[A,B] = -\frac{1}{q} \res (A[B,\log Q]). 
		\end{equation}
		\item { 
		Let us consider a family $A_t$ of classical pseudo-differential operators of constant order, and  
		a family $Q_t$ of weights of constant order $q$, both of which are differentiable with respect to the 
		Kontsevich and Vishik Fr\'echet structure	on $Cl(M,E)$. } Then, 
		\begin{equation}\label{deriv} 
		\frac{d}{dt} \left(tr^{Q_t}A_t\right) = tr^{Q_t} \left(\frac{d}{dt} A_t\right) -
		\frac{1}{q} \res \left( A_t (\frac{d}{dt}\log Q_t) \right).
		\end{equation}
		\item If C is a classical elliptic
		injective operator or a diffeomorphism, and $A$ is a classical pseudodifferential operator,
		$tr^{C^{-1}QC}\left( C^{-1}AC \right)$ is well-defined and equals $\tr^QA$. 
		\item Finally, $$\tr^QA = \overline{\tr^{Q}A^*}.$$
	\end{itemize}
\end{Proposition}
In this proposition we have followed \cite{CDMP}, and \cite{Ma2016} for the third point.

\smallskip

We have stated that $\tr^Q$ is not a true trace; however, the renormalized trace of the bracket satisfies
some interesting properties which we state following \cite{Ma2006}.

\begin{Definition} \label{d8}
	Let $E$ be a vector bundle over $M$ let $Q$ a weight and let $a \in \Z$. We define :
	$$ {\mathcal A}^Q_a=\{B \in Cl(M,E) : [B,\log Q] \in Cl^a(M,E)\}.$$
\end{Definition}

\begin{Theorem} \label{t1}
	\begin{item}
		(i) ${\mathcal A}^Q_a \cap Cl^0(M,E) $ is an subalgebra of $Cl(M,E)$
		with unit.
	\end{item}
	\begin{item}
		(ii) Let $B \in Ell^*(M,E)$, $B^{-1}{\mathcal A}^Q_aB = {\mathcal A}^{B^{-1}QB}_a.$
	\end{item}
	\begin{item}
		(iii) Let $A\in Cl^b(M,E)$, and $B \in {\mathcal A}^Q_{-b-2}$,
		then $\tr^Q[A,B]=0.$ As a consequence, 
		$$\forall (A,B) \in Cl^{-\infty}(M,E) \times Cl(S^1,V), \quad \tr^Q[A,B]=0\; .$$
	\end{item}
\end{Theorem}

We are ready to state the properties of $tr^Q$ that make odd-class pseudodifferential operators an interesting
arena for infinite-dimensional mechanics.

\begin{Theorem} \label{commutator}
	Let $A,B \in Cl(M,E)$ and let $Q$ be an odd-class weight of even order, e.g. $Q = \Delta.$
	\begin{itemize}
		\item  If $(A,B) \in Cl_{odd}(M,E) \times Cl_{odd}(M,E)$, and if $M$ is odd dimensional, 
		$$\tr^Q([A,B])=0\; .$$
		\item  If $(A,B) \in Cl_{even}(M,E)\times Cl_{odd}(M,E),$ and if $M$ is even dimensional dimensional, 
		$$\tr^Q([A,B])=0\; .$$
	\end{itemize}
\end{Theorem}
\begin{proof} 
The first item is due to Kontsevich and Vishik, see \cite{KV1,KV2}. We sketch a proof of this 
theorem following \cite{Scott}:

If $Q$ and $B$ are odd class, with $Q$ of even order, as in the statement of the theorem, 
$[B,\log Q] \in Cl_{odd}(M,E).$ Thus, 
\begin{itemize}
	\item  If $A \in Cl_{odd}(M,E),$ then $A[B, \log Q] \in Cl_{odd}(M,E).$
	\item  If $A \in Cl_{even}(M,E),$ then $A[B, \log Q] \in Cl_{even}(M,E).$
\end{itemize}
Symmetry properties show that in both cases $$\int_{|\xi|=1} \sigma_{-m}\left(A[B, \log Q]\right) =0\; ,$$
and the result follows by applying the local formula for the Wodzicki residue. 
\end{proof}

\begin{Corollary}  \label{adinvariance}
{ Let $Q = f(\Delta)$ in which $f$ is any analytic function such that $Q$ is a weight,} 
and assume that $A,B$ and $C$ are classical pseudodifferential operators either in the odd-class 
or in the even-class. If the product $ABC$ is odd class and $M$ is odd dimensional, or if the product $ABC$ is 
even class and $M$ is even dimensional, then 
$$ tr^{ Q}(ABC) = tr^{ Q}(CAB) = tr^{ Q} (BCA)\; .$$
\end{Corollary}

\section{Renormalized traces determine non-degenerate pairings}

In this short but crucial section we give an extension of a result from \cite[Section 3.2]{Ma2020-4} which 
connects the foregoing discussion with Hermitian geometry. We remark that we do not assume that $M$ is 
an odd dimensional manifold, so that $tr^Q$ is not a priori a true trace.

\begin{Theorem} \label{Th:Hermitian}
We consider a weight $Q$ and a fixed classical pseudodifferential operator $Q_0 \in Cl(M,E).$

	\begin{enumerate}
		\item The sesquilinear map
		$$(.,.)_{Q,Q_0} :  (A,B) \in Cl(M,E) \times Cl(M,E) \mapsto \tr^Q\left(AQ_0B^*\right)$$
		is non-degenerate if and only if $Q_0$ is injective. 
		\item Moreover, if $Q_0$ is self-adjoint, then  $(.,.)_{Q,Q_0}$ is Hermitian, this is,
		$$  (B,A)_{Q,Q_0} = \overline{(A,B)}_{Q,Q_0} \; .$$
	\item
	As a consequence, the Hilbert-Schmidt positive definite Hermitian product 
	$$ \left( A,B\right)_{HS} = \tr\left(AB^*\right)$$
	which determines a positive definite metric on $Cl^{-1- dim M}(M,E)$,  extends to a Hermitian
	 form
	$$(\cdot\, , \cdot )_\Delta = (\cdot\, , \cdot )_{\Delta,Id}$$ 
	\begin{itemize}
		\item which is a non-degenerate form on $Cl(M,E)$ 
		\item whose real part defines a $(\R-)$ bilinear, symmetric non-degenerate form on $Cl(M,E)$. 
	\end{itemize}
\end{enumerate}
The same properties hold true if we replace $Cl(M,E)$ by $Cl^0(M,E)$ in statements $(1)$, $(2)$, $(3)$.
\end{Theorem}
\begin{proof}
	\begin{enumerate}
		\item 	First, let us assume that $Q_0$ is not injective. Let $y \neq 0 \in Ker Q_0$ and let $A= p_y$ 
		be $L^2$ orthogonal projection on the 1-dimensional vector space spanned by $y.$ Then $AQ_0=0$, and  
		$\forall B \in Cl(M,V)$ we have $(A,B)_{Q,Q_0} = 0$, so that $(\cdot\, , \cdot )_{Q,Q_0}$ is
		degenerate. 
		
		Let us now assume that $Q_0$ is injective. Then, $\forall A \neq 0 \in Cl(M,E)$, $AQ_0 \neq 0.$ 
		The formula $(A,B)_{Q,Q_0}=\tr^Q\left(A Q_0 B^*\right)$ certainly defines a sesquilinear form, 
		let us  prove that it is non-degenerate. Let $A \in Cl(M,E),$ and let 
		$u \in C^\infty(M,E) \cap \left(Im AQ_0 - \{0\} \right)$. We assume that $u$ is the image of a 
		function $x$ such that $||x||_{L^2} = 1, $ and we let $p_x$ be the $L^2-$ orthogonal projection on the 
		$\mathbb{C}-$vector space spanned by $x.$ Finally, we also let $(e_k)_{k \in \N}$ be an orthonormal base 
		with $e_0=x.$ We will analyse $\phi(s)=tr(AQ_0B^*Q^{-s})$ when $s$ is ``large enough'', 
		and then we will pass to the meromorphic continuation. For 
		$\mathfrak{Re}(s) \geq {(2ord(A) + 2 ord Q_0 + 1 + dim M)}{/q},$ we observe that the operators 
		$$AQ_0 \left(AQ_0p_x\right)^*Q^{-s}\; ,$$ 
	    $$AQ_0 \left(AQ_0p_x\right)^*Q^{-s/2}\; ,$$ 
	    $$Q^{-s/2}\; , $$ 
		$$\left(AQ_0p_x\right)^*Q^{-s/2}\; ,$$ and 
		$$Q^{-s/2}AQ_0$$ 
		are Hilbert-Schmidt class. We recall that for 
		Hilbert-Schmidt class operators $U$ and $V,$ $UV$ is trace class and $tr(UV) = tr(VU).$ Thus, applying
		commutation relations of the usual trace of trace-class operators, we obtain the following: 
		\begin{eqnarray*} 
		\phi(s) =\tr\left(AQ_0 \left(AQ_0p_x\right)^*Q^{-s}\right)
		   & = & \tr\left(Q^{-s/2}AQ_0 \left(AQ_0p_x\right)^*Q^{-s/2}\right) \\ 
		   & =& \tr\left( \left(AQ_0p_x\right)^*Q^{-s/2}.Q^{-s/2}AQ_0\right) \\ 
		   & = &\tr\left( \left(AQ_0p_x\right)^*Q^{-s}AQ_0\right)\; .
		   \end{eqnarray*}
		   Now we simplify this expression in order to show that the meromorphic continuation  of 
		   $\phi(s)$ has no poles and a non-zero value at $s=0$. The meromorphic continuation to $\mathbb{C}$ of 
		$s \mapsto \tr\left( \left(AQ_0p_x\right)^*Q^{-s}AQ_0\right)$ exists and it coincides with the 
		meromorphic continuation of $s \mapsto \phi(s)$; in 
		particular they coincide at $s=0.$ Moreover,
		\begin{eqnarray*}
			\tr\left( \left(AQ_0p_x\right)^*Q^{-s} AQ_0\right) 
			& = & \sum_{k \in \mathbb{N}}  (\left(AQ_0p_x\right)^*Q^{-s} AQ_0 e_k, e_k)_{L^2} \\
			& = &  \sum_{k \in \mathbb{N}}  (Q^{-s} AQ_0 e_k, AQ_0p_x e_k)_{L^2}\\
			& = & (Q^{-s} AQ_0 x, AQ_0x)_{L^2}\\
			& = & (Q^{-s/2}u,Q^{-s/2}u)_{L^2} \; .
		\end{eqnarray*}
		Since $\lim_{s \rightarrow 0} Q^{-s/2} = Id$ for weak convergence, the limit of the last term is 
		$||u||_{L^2}^2\neq 0.$
		The operator $AQ_0p_x$ is a smoothing (rank 1) operator and hence it belongs to $Cl(M,E),$ 
		which ends the proof.
		
	\item Let $(A,B)\in Cl(M,E). $ We calculate directly using Proposition \ref{p6}:  
	$$ (B,A)_{Q,Q_0}=\tr^Q(BQ_0A^*) = \tr^Q\left((AQ_0B^*)^*\right) 
	= \overline{\tr^Q\left(AQ_0B^*\right)} = \overline{(A,B)}_{Q,Q_0} \; .$$
	
	\item  It follows from the two previous items.
\end{enumerate}
We finish the proof by remarking that our foregoing arguments hold true when considering only bounded classical 
pseudodfferential operators. 
\end{proof}

\begin{remark} \label{nopos}
	We remark that if $Q_0$ is self-adjoint {and injective}, the polarization identity 
$$ \mathfrak{Re}(A,B)_{Q,Q_0} = \frac{1}{4}\left[ (A+B,A+B)_{Q,Q_0} - (A-B,A-B)_{Q,Q_0} \right] $$
	implies that $\mathfrak{Re}(.,.)_{Q,Q_0}$ 
	is a symmetric and non-degenerate { real-valued} bilinear form. This bilinear form is not positive-
	definite, see a direct calculation for $M = S^1$ in \cite[Section 3]{Ma2020-4}.
\end{remark}

\section{Rigid body equations}

In this and the next section we work with the regular Fr\'echet Lie group of odd-class pseudodifferential
operators $Cl_{odd}^{0,*}(M,E)$ and its Lie algebra $Cl_{odd}^{0}(M,E)$. Our main claim is that this Lie group is a 
non-trivial differential geometric framework on which we can pose equations of mechanics in the spirit of Arnold, 
see \cite{fourier}. Our main references for this section are \cite{GPoR,Holm,MR2016} and \cite{T}. 

We remark that in Subsection 4.2 (and also in Section 5) we consider pseudo-Riemannian metrics on 
$Cl_{odd}^{0,*}(M,E)$ induced by twisting the non-degenerate bilinear forms constructed in Section 3. Our metrics
are defined using {\em right} translations, see e.g. \cite{Holm}. This convention forces us to re-define the $ad_X$
morphism on $Cl_{odd}^{0}(M,E)$ as $ad_X Y = - [X,Y] = -(XY-YX) = YX-XY = [Y,X] .$

 \subsection{The Hamiltonian construction}
We consider the trace $tr^\Delta$ on the regular Lie algebra $Cl_{odd}^{0}(M,E)$ and the pairing {
$$
\left< A , B \right> = (A,B)_{\Delta,Q_0} = tr^\Delta (A\,Q_0\,B^*) \; ,
$$}
in which $Q_0$ is injective (and hence the pairing is non-degenerate) and self-adjoint (and hence the pairing is 
Hermitian), and we also consider its real part $\mathfrak{Re}\left< A , B \right>.$ 
 {\em We also assume, here and hereafter, that the following constraints on $Q_0$ hold}:
	
	\smallskip
	
	\begin{tabular}{|c|c|}
		\hline
		if $M$ is... & then $Q_0$ is... \\
		\hline
		odd dimensional & an odd-class operator\\
		\hline
		even dimensional &  an even-class operator \\
		\hline
	\end{tabular}

\smallskip

\noindent In this way we are sure that the commutation relations for $\tr^Q(AQ_0B^*)$ appearing in Theorem \ref{commutator} and Corollary \ref{adinvariance} hold. Trivially, if $Q_0$ is injective, self-adjoint and smoothing (e.g. $Q_0 = e^{-\Delta}$), these conditions are fulfilled for any manifold $M.$

The next lemma is a direct 
consequence of Theorem \ref{Th:Hermitian}:

%Theorem \ref{commutator} and Corollary \ref{adinvariance} imply the following crucial result:

\begin{Lemma}
{ Let us assume that $Q_0$ is an injective and self-adjoint classical pseudodifferential operator. }

\begin{enumerate}
	\item The $\mathbb{C}$-valued pairing { $\left< A , B \right> =  tr^\Delta (A\,Q_0\,B^*)$ on
	$Cl^0_{odd}(M,E)$ is Hermitian} and non-degenerate.
	%\item In addition, if $Q_0=1$, then $\left< A , B \right>$ is infinitesimally Ad-invariant, this is,
	%\begin{equation} \label{inf-ad}
	%	\left< [A , B],C \right> = \left<[C, A ], B \right>
	%\end{equation}
	% for all $A,B,C  \in Cl_{odd}^{0}(M,E)$.
	\item The real-valued pairing $\,\mathfrak{Re}\left< A , B \right>$ is bilinear, symmetric and non-degenerate 
	for any choice 
	of self-adjoint and injective operator { $Q_0 \in Cl(M,E).$}
\end{enumerate}
\end{Lemma}

\smallskip

This lemma allows us to consider the {\em regular dual space} of $Cl_{odd}^{0}(M,E)$, namely,
$$
Cl_{odd}^{0}(M,E)' = \left\{ \mu \in L (Cl_{odd}^{0}(M,E) , \mathbb{C} ) : \mu = \left< A , \cdot \right>
\mbox{ for some } A \in Cl_{odd}^{0}(M,E) \right\}\; .
$$
%and smooth functions on $Cl_{odd}^{0}(M,E)'$.  
We can equip $Cl_{odd}^{0}(M,E)'$ with a Fr\'echet structure simply by transferring the
structure of $Cl_{odd}^{0}(M,E)$, since there is a bijection between $Cl_{odd}^{0}(M,E)'$ and 
$Cl_{odd}^{0}(M,E)$.

We consider smooth polynomial functions on $Cl_{odd}^{0}(M,E)'$ of the form
\begin{equation} \label{Pol}
f(\mu) = \sum_{k=0}^n a_k\, tr^\Delta(P^k)\; ,
\end{equation}
in which $a_k \in \mathbb{C}$ and $P$ is determined by the equation $\mu = \left< P , \cdot \right>$. 

If $f$ is such a smooth function on $Cl_{odd}^{0}(M,E)'$, we define the functional derivative of $f$, 
$\delta f/\delta \mu \in Cl^0_{odd}(M,E)$, as in (\ref{gradiente}), that is, via the equation
$$
\left< \nu\, , \frac{\delta f}{\delta \mu} \right>  = (d \, f)_\mu (\nu) = \left.\frac{d}{d\epsilon}
\right|_{\epsilon=0} f(\mu + \epsilon \nu) \; ,
$$
and we equip $Cl_{odd}^{0}(M,E)'$ with a Poisson bracket (which acts on polynomial
functions), see Equation (\ref{lie-poisson0}), \cite{BP}, and also \cite{ER2013,MR2016}. We set:
{
    \begin{equation}
    \{f \, , \, g \}(\mu) = +  \left< \mu \, \; , \; \left[
    \frac{\delta \, f}{\delta \, \mu} \, ,
    \, \frac{\delta \, g}{\delta \, \mu} \right] \right>   \label{lie-poisson}
    \end{equation}
for smooth functions $f, g : Cl_{odd}^{0}(M,E)' \rightarrow \mathbb{C}$ and $\mu \in Cl_{odd}^{0}(M,E)'$.
The plus sign is due to our right translation convention, see \cite[Remark 9.12]{Holm}. 
Next, let us fix a smooth function $H : Cl_{odd}^{0}(M,E)' \rightarrow \mathbb{C}$. The bracket 
(\ref{lie-poisson}) determines a derivation $X_H$ on functions $f : Cl_{odd}^{0}(M,E)' \rightarrow \mathbb{C}$ via
a prescription as in (\ref{vector}), this is, $X_H (\mu) \cdot f = \{ H , f \}(\mu)$ 
for all $\mu \in Cl_{odd}^{0}(M,E)'$; we can then pose Hamilton's equations
    \begin{equation} \label{he}
    \frac{d }{dt} (f \circ \mu) = X_H(\mu) \cdot f 
    \end{equation}
on  $Cl_{odd}^{0}(M,E)'$. For $\mu(t) = \left< P(t) , \cdot \right> \in Cl_{odd}^{0}(M,E)'$
   they become
    $$
  \left<  \frac{d \mu}{d t} , \frac{\delta \, f}{\delta \, \mu} \right> = + 
  \left< \mu\, \; , \; \left[\frac{\delta\, H}{\delta\, \mu}\, , \frac{\delta\, f}{\delta\, \mu}\right]\right>\; ,
    $$
    this is,
\begin{equation} \label{lax1}
     \left<  \frac{d P}{d t} , Q \right> = + 
  \left< P \, \; , \; \left[ \frac{\delta \, H}{\delta \, \mu}\, , \, Q \right] \right>
\end{equation}
    for $\displaystyle Q = \frac{\delta \, f}{\delta \, \mu} \in  Cl_{odd}^{0}(M,E)$.

\smallskip

This is a ``weak version" of the Euler equation appearing in Berezin and Perelomov's paper \cite{BP}. In our 
Hermitian context we do not have infinitesimal Ad-invariance, and so we obtain (\ref{lax1}) instead of a
standard equation such as 
$$
   \frac{d \, P}{d t} =  \left[ P\, ,  \frac{\delta \, H}{\delta \, \mu} \right] \; ,
$$
see \cite[Equation (8)]{BP}. For example, if we take 
$\mu = \left< P , \cdot \right>$ and $\displaystyle H_{k}(\mu) =  tr^\Delta \left( P^{k} \right)$,
    $k=1,2,3,\cdots ,$ we can easily check that {(assuming existence of $Q_0^{-1}$)}
    $$    \frac{\delta H_{k}}{\delta \mu} = k Q_0^{-1} (P^*)^{k-1} \; , $$
    and Equation (\ref{lax1}) on $Cl_{odd}^{0}(M,E)$ become 
    \begin{equation}
      \left<  \frac{d P}{d t} , Q \right> =  \,k\,
  \left< P \, \; , \; \left[\, Q_0^{-1} (P^*)^{k-1} \, , \, Q \, \right] \right> \; .  \label{kp}
    \end{equation}
    
    \begin{remark} \label{nonlocal}
The presence of the operator $Q_0^{-1}$ requires us to be careful. An injective self-adjoint pseudodiferential 
operator $Q_0$ has: 
 \begin{enumerate}
 	\item an inverse $Q_0^{-1}$ which is itself an injective self adjoint pseudodifferential operator 
 	\textit{if and only if $Q_0$ is not smoothing}
 	\item an inverse $Q_0^{-1}\notin Cl(M,E) $ \textit{if and only if $Q_0$ is smoothing.}  
 \end{enumerate}   		
The second case is the one which needs more attention. Indeed, if $Q_0$ is smoothing, e.g. $Q_0 = e^{-\Delta},$ 
then its formal symbol vanishes. This explains why its inverse cannot be a pseudodifferential operator. However, 
$Q_0$ is a self-adjoint injective compact operator. Hence, via spectral analysis, it is easy to define $Q_0^{-1}$ 
which is an unbounded operator with $L^2-$dense domain in $C^\infty(M,E).$ Therefore, Equation $(\ref{kp})$ is 
always well-stated. 
    \end{remark}

The foregoing equations are equations on the (regular) dual of the Lie algebra $Cl^0_{odd}(M,E)$. We can work 
directly on the Lie algebra $Cl^0_{odd}(M,E)$ and we can use more general pairings if we proceed as follows. 

\smallskip

We assume that there exists an operator 
$\mathbb{A} : Cl_{odd}^{0}(M,E) \rightarrow Cl_{odd}^{0}(M,E)$ such that the new pairing
$$
\left< X,Y \right>_\mathbb{A} = \left< X , \mathbb{A} (Y) \right>
$$
is Hermitian and non-degenerate. We think of $\mathbb{A}$ as a twist of our previous pairing or, motivated by 
\cite{fourier,Holm,KW}, see also \cite{GPoR}, as an ``inertia operator". 
We also consider the real part of $\left< \cdot \,,\, \cdot \right>_\mathbb{A}$,
$$
\mathfrak{Re}\left< X,Y \right>_\mathbb{A} = \mathfrak{Re}\left< X , \mathbb{A} (Y) \right> \; 
$$ 
for $X,Y \in Cl_{odd}^{0}(M,E)$. Since $\left< \cdot \, ,\, \cdot \right>_\mathbb{A}$ is Hermitian and 
non-degenerate, this new pairing is a symmetric and non-degenerate real-valued bilinear form which 
allows us to consider (in view of Remark \ref{nopos}) pseudo-Riemannian 
geometry. In order to do so, we define a new adjoint map as 
\begin{equation}  \label{ada}
\mathfrak{Re}\left< [X , Y] , Z \right>_\mathbb{A} = 
- \mathfrak{Re}\left< Y , ad_\mathbb{A}(X)\, Z \right>_\mathbb{A} =  
- \mathfrak{Re} \left<  ad_\mathbb{A}(X)\, Z , Y \right>_\mathbb{A} \; ,
\end{equation}
so that $ad_\mathbb{A}(X)$ is the adjoint of $ad_X$ in accordance with our sign convention, see also 
\cite[Section 2]{GPoR}. We compute $ad_\mathbb{A}$ explicitly as follows:
\begin{eqnarray*}
- \mathfrak{Re}{\left< Y , ad_\mathbb{A}(X)\, Z \right>}_\mathbb{A} & = & 
 -   \mathfrak{Re}\left< ad_X Y , Z  \right>_\mathbb{A} = 
\mathfrak{Re} \left< [X ,Y] , \mathbb{A}(Z)  \right> \\
 & = & \mathfrak{Re}\, tr^\Delta ([X,Y]\,Q_0\, \mathbb{A}(Z)^* ) \\
 & = & \mathfrak{Re}\, tr^\Delta\left(XY Q_0 \mathbb{A}(Z)^* - YX Q_0\,\mathbb{A}(Z)^*\right) \\
 & = & \mathfrak{Re}\, tr^\Delta\left( YQ_0 \mathbb{A}(Z)^*X - YX Q_0\,\mathbb{A}(Z)^*\right) \\
 & = & \mathfrak{Re}\, tr^\Delta\left(Y (Q_0\, \mathbb{A}(Z)^*X - X Q_0\,\mathbb{A}(Z)^*)\right) \\
   & = & \mathfrak{Re}\, tr^\Delta\left(Y \left[Q_0\, \mathbb{A}(Z)^*,X\right]\right) \\
    & = & \mathfrak{Re}\, tr^\Delta\left(Y Q_0 Q_0^{-1} \left[Q_0\, \mathbb{A}(Z)^*,X\right]\right) \; .
\end{eqnarray*}    

We set 
\begin{equation} \label{auxA}
- \mathbb{A}(R)^* = Q_0^{-1} \left[Q_0\, \mathbb{A}(Z)^*,X\right]\; .
\end{equation}
Then, $- \mathfrak{Re}{\left< Y , ad_\mathbb{A}(X)\, Z \right>}_\mathbb{A} = 
-\mathfrak{Re}\, tr^\Delta\left(Y\, Q_0\, \mathbb{A}(R)^* \right)
= - \mathfrak{Re}\,\left< Y , R \right>_\mathbb{A}\,$, and therefore $ad_\mathbb{A}(X) Z = R$. We compute $R$ 
quite easily. Equation (\ref{auxA}) implies
$$
\mathbb{A}(R) = \left[X, Q_0\, \mathbb{A}(Z)^*\right]^* Q_0^{-1} = \left[ \mathbb{A}(Z)\, Q_0 , X^* \right] 
Q_0^{-1}  \; ,
$$
and so we conclude that
\begin{equation} \label{adA}
ad_\mathbb{A}(X)\, Z = \mathbb{A}^{-1}\left( \left[ \mathbb{A}(Z)\, Q_0,X^*\right]Q_0^{-1} \right) {
= - \mathbb{A}^{-1}\left( [ad_{\mathbb{A}(Z) Q_0} X^*] Q_0^{-1} \right)\; .}
\end{equation}

\subsection{{ Euler-Lagrange} equations}
{
Now we use $ad_\mathbb{A}$ and the bilinear form $\mathfrak{Re}\left<\cdot , \cdot \right>_\mathbb{A}$ 
to write down equations of motion on the Lie group $Cl_{odd}^{0,*}(M,E)$. Our equations are Euler-Lagrange
equations arising from a natural action functional. We follow, roughly, Taylor's lecture notes \cite{T}.
}

\smallskip

We set $\left< \cdot | \cdot \right> = \mathfrak{Re}\left< \cdot , \cdot \right>_\mathbb{A}$ just to simplify our 
notation. First of all, we extend the symmetric and non-degenerate bilinear form $\left< \cdot | \cdot \right>$ to 
a pseudo-Riemannian metric on $G=Cl^{0,*}_{odd}(M,E)$ via right translation:
\begin{equation} \label{metric1}
g(P)(V,W) = \left< T_P R_{P^{-1}} V | T_P R_{P^{-1}} W \right> \; ,
\end{equation}
in which $P \in G$, $W,V \in T_P G$, and $R_{P^{-1}}$ is right translation. {We simplify this 
expression using the identification $V = (P+\epsilon Q_1)'(0)$ and $W = (P+\epsilon Q_2)'(0)$
for $Q_1, Q_2 \in Lie(G)$; }
we obtain 
$$
g(P)(V,W) = \left< Q_1 P^{-1} | Q_2 P^{-1} \right> \; .
$$

Now we set up the {kinetic energy} Lagrangian functional on curves in $G$,
$$
I[P(t)] = \int_a^b g(P(t))(\dot{P}(t),\dot{P}(t)) dt = \int_a^b \left< \dot{P}(t) P(t)^{-1} | 
\dot{P}(t) P(t)^{-1} \right> dt \; ,
$$
in which $\dot{P}(t)$ is now considered an an element of $Lie(G)$ for each $t$, and we find the corresponding
 equation for critical points of $I$. We assume that $t \mapsto P(t)$ is a critical, and we deform 
this curve slightly via $P(t) \mapsto P(t) + \epsilon \eta(t) Q$, with $\eta(a)=\eta(b)=0$ and $Q \in Lie(G)$  
{in such a way that that this deformed curve lies in $G$ equipped with its Fr\'echet topology (recall that 
$Cl^{0,*}(M,E)$ is open in $Cl^{0,*}(M,E)$)}.  Because $t \mapsto P(t)$ is critical, we have
$$
\left. \frac{d}{d\epsilon}\right|_{\epsilon=0} I [P(t) + \epsilon \eta(t) Q] = 0\; .
$$
Hereafter we omit $t$ dependence for clarity. We have: 
$$
\left. \frac{d}{d\epsilon}\right|_{\epsilon=0} I [P(t) + \epsilon \eta(t) Q] =
\int_a^b \left. \frac{d}{d\epsilon}\right|_{\epsilon=0} 
\left< (\dot{P}+\epsilon \dot{\eta} Q) (P+\epsilon \eta Q)^{-1} |  (\dot{P}+\epsilon \dot{\eta} Q) 
(P+\epsilon \eta Q)^{-1} \right> dt = 0 \; ,
$$
this is, 
$$
\int_a^b \left< \dot{\eta} Q P^{-1} - \dot{P} P^{-1} \eta Q P^{-1} |  \dot{P} P^{-1} \right> dt = 
\int_a^b \dot{\eta} \left<  Q P^{-1} | \dot{P}P^{-1} \right> dt - \int_a^b \eta \left< \dot{P} P^{-1} Q P^{-1} |  
\dot{P} P^{-1} \right> dt
=0 \; .
$$
We integrate by parts and use the boundary conditions for $\eta$; we obtain
$$
- \int_a^b \eta \left<  Q P^{-1} | \dot{P}P^{-1} \right>^{\cdot} dt - 
\int_a^b \eta \left< \dot{P} P^{-1} Q P^{-1} |  \dot{P} P^{-1} \right> dt = 0 \; ,
$$
this is,
$$
\int_a^b \eta \left\{ \left< Q P^{-1} \dot{P} P^{-1} | \dot{P} P^{-1} \right> -
\left<  Q P^{-1} | (\dot{P}P^{-1})^{\cdot} \right> \right\} dt 
- \int_a^b \eta \left< \dot{P} P^{-1} Q P^{-1} |  \dot{P} P^{-1} \right> dt = 0 \; .
$$
Since $\eta(t)$ is arbitrary, we find the equation of motion 
\begin{equation} \label{geodesic1}
 \left< Q P^{-1} \dot{P} P^{-1} | \dot{P} P^{-1} \right> -
\left<  Q P^{-1} | (\dot{P}P^{-1})^{\cdot} \right> 
- \left< \dot{P} P^{-1} Q P^{-1} |  \dot{P} P^{-1} \right> = 0 
\end{equation}
in which $Q$ is an arbitrary element of $Lie(G)$. 

Since $\dot{P} P^{-1}$ and $Q P^{-1}$ belong to $Lie(G)$, we can write $\dot{P} P^{-1}=X$ and $Q P^{-1}=W$ for
$X,W \in Lie(G)$. Equation (\ref{geodesic1}) becomes
$$
\left< W X | X \right> - \left<  W | \dot{X} \right> 
- \left< X W | X \right> = 0 
$$
for all $W \in Lie(G)$, this is,
\begin{equation} \label{geodesic2}
\left< [W , X] | X \right> = \left<  W | \dot{X} \right> 
\end{equation}
for all $W \in Lie(G)$. As pointed out in \cite{T}, if we solve for $X$ in (\ref{geodesic2}), the curve 
$P(t)$ is recovered via $\dot{P}(t) = X(t) P(t)$. Thus, Equation (\ref{geodesic2}) ---an equation posed
on $Lie(G)$--- determines a family of curves on $G$. It remains to find a ``strong" formulation of 
(\ref{geodesic2}). We go back to the notation used in Subsection 4.1. Equation (\ref{geodesic2}) becomes
$$
 \mathfrak{Re} \left< ad_X W , X \right>_\mathbb{A} = 
 \mathfrak{Re} \left< W , \dot{X} \right>_\mathbb{A} \; ,
$$
and therefore, using the operator $ad_\mathbb{A}$ we obtain
$$
\mathfrak{Re} \left< W , ad_\mathbb{A}(X) X \right>_\mathbb{A} = 
\mathfrak{Re} \left< W , \dot{X} \right>_\mathbb{A} \; .
$$
Non-degeneracy of the inner product $\mathfrak{Re}\left< \cdot , \cdot \right>_\mathbb{A}$ implies that 
$X(t) \in Cl^0_{odd}(M,E)$ satisfies the
non-linear equation
\begin{equation} \label{lax2}
	\frac{d}{dt} X = ad_\mathbb{A}(X(t)) X(t) \; . 
\end{equation}
We note the formal similarity between (\ref{geodesic2}) and the Hamiltonian equation (\ref{lax1}). Due to this
fact, we naturally call (\ref{geodesic2}), or (\ref{lax2}), the Euler equation on $Cl^0_{odd}(M,E)$. We have
proven the following theorem:

\begin{Theorem}
The Euler equation 
\begin{equation} \label{lax3}
	\frac{d}{dt} X = \mathbb{A}^{-1}\left( \left[ \mathbb{A}(X)\, Q_0,X^*\right]Q_0^{-1} \right) 
\end{equation}
on $Cl^0_{odd}(M,E)$, is the Euler-Lagrange equation of the kinetic energy action functional on the Fr\'echet Lie 
group $Cl^{0,*}_{odd}(M,E)$ equipped with the pseudo-Riemannian metric $(\ref{metric1})$.
\end{Theorem}

{
Equation (\ref{lax2}) is formally analog to the Euler equation posed on a Lie group $G$ equipped with a
{\em Riemannian} metric. In  this Riemannian case, our foregoing computations translate {\em mutatis mutandis} into 
the well-known fact that Euler equations determine geodesics on $G$, see for instance \cite{Holm,KW,T} and 
references therein. 
}

\begin{remark}
If we take $Q_0 = Id$ and we pose Equation $(\ref{lax3})$ on the subgroup of
self-adjoint operators, we obtain 
$$
\mathbb{A}\left( \frac{d}{dt} X \right) = \left[ \mathbb{A}(X) ,X \right] = - ad_{\mathbb{A}(X)} X\; ,
$$
an equation that looks exactly as the classical Euler equation in $so(3)$, see \cite[Theorem 7.2]{Holm}.
\end{remark}

{
\subsection{Remarks on integrability}

Motivated by Manakov's observation on the integrability of the rigid body equation, see \cite{Ma}, we state:

\begin{Proposition}  \label{propint}
The Euler equation 
\begin{equation} \label{lax4}
	\frac{d}{dt} X = \mathbb{A}^{-1}\left( \left[ \mathbb{A}(X)\, Q_0,X^*\right]Q_0^{-1} \right) 
\end{equation}
on $Cl^0_{odd}(M,E)$ is equivalent to the Lax pair equation
\begin{equation} \label{lax5}
\frac{d}{dt} (\mathbb{A}(X)Q_0 + \xi J^2) = \left[ \mathbb{A}(X)\, Q_0 + \xi J^2 , X^* + \xi J \right]
\end{equation}
in which $\xi$ is a complex parameter and $J$ is an operator satisfying $\mathbb{A}(X) Q_0 J = J \mathbb{A}(X) Q_0$ 
and $X^* J^2 = J^2 X^*$.
\end{Proposition}
\begin{proof}
Equation (\ref{lax4}) can be written as
$$
\frac{d}{dt} \left( \mathbb{A}(X)Q_0 \right) =  \left[ \mathbb{A}(X)\, Q_0,X^*\right]\; ,
$$
and this equation is  equivalent to Equation (\ref{lax5}) for arbitrary values of $\xi$.
\end{proof}

We interpret this proposition as saying that our Euler equation (\ref{lax4}) posed on $Cl^{0}_{odd}(M,E)$ admits a
Lax pair formulation and it is therefore integrable. We can also prove that 
$$
I_k = tr^\Delta \left( (\mathbb{A}(X)Q_0 + \xi J^2)^k \right)
$$
is conserved along solutions to (\ref{lax5}) for arbitrary values of $\xi$ and $k \geq 1$. Indeed, it is easy to 
check that
$$
\frac{d}{dt}\left( (\mathbb{A}(X)Q_0 + \xi J^2)^k \right) = 
tr^\Delta \left( [(\mathbb{A}(X)Q_0 + \xi J^2)^k , X^* + \xi J ] \right) = 0
$$
on solutions to (\ref{lax5}), and therefore expansion of $I_k$ in powers of $\xi$ yields integrals of
motion for (\ref{lax4}). Since the conditions on $J$ appearing in Proposition \ref{propint} imply 
that the operators $\mathbb{A}(X)Q_0$ and $\xi J^2$ commute, we can easily obtain explicit expressions for these 
integrals by expanding $I_k$. We present an example in Section 6.
%$$
%I_k = \sum_{j=0}^k \left( \begin{array}{c} h \\ j \end{array} \right) tr^\Delta \left( [\mathbb{A}(X)Q_0]^{k-j}
%J^{2j} \right) \xi^j\; ,
%$$
}

\section{{Pseudo-Riemannian Geometry on $Cl^{0,*}_{odd}(M,E)$} }

{
In this section we review some basics facts of the pseudo-Riemannian geometry of the regular Fr\'echet group
$Cl^{0,*}_{odd}(M,E)$, motivated by Arnold's classical paper \cite{fourier}. We fix an inertia operator 
$\mathbb{A}$ and we consider the pseudo-Riemannian metric on $Cl^{0,*}_{odd}(M,E)$ induced by right translation
of the non-degenerate and symmetric bilinear form $\mathfrak{Re}\left< \cdot\, ,\, \cdot,  \right>_\mathbb{A}\,$,
see Equation (\ref{metric1}). 
}

We note that there exist some difficulties in describing the whole space of connection 1-forms 
$$\Omega^1(Cl^{0,*}_{odd}(M,V),Cl^{0}_{odd}(M,V))\; .$$ Indeed, to our knowledge, the space of smooth linear maps 
acting on $Cl^{0}_{odd}(M,V)$ is actually not well-understood. 
In the classical setting of a \textit{Riemannian} (e.g. finite dimensional, or Hilbert) Lie group $G$ with 
Lie algebra $\mathfrak g$, the Levi-Civita connection 1-form (i.e metric-compatible and torsion-free) reads as 
%$$ \forall (X,Y) \in \mathfrak{g}^2, \quad 
$$\theta_XY = \frac{1}{2} \left\{ad_XY - ad_X^*Y - ad_Y^*X\right\}\; ,$$
in which $ad^*$ is the adjoint of $ad$ with respect to the metric of $G$ {and $X,Y$ are {\em left 
invariant} vector fields, see \cite[Proposition 1.7]{Freed88}. }
It is possible to go beyond this well-known result, and extend it to (pseudo-)Riemannian right-invariant metrics, 
if an adjoint for $ad$ is known. {Formal calculations have been already carried out, see for example
the classical Arnold's paper \cite{fourier} or \cite[Section 2]{GPoR} and references therein, but in the context of 
pseudodifferential operators, finding a rigorous {({\em i.e.} truly smooth)} adjoint of the adjoint map, as 
described in \cite{Ma2020-4}, remains a difficult task. We can bypass this difficulty here,} since we already have 
$ad_\mathbb{A}$ at our disposal. 

\begin{Theorem}   \label{oneform}
	Let $(X,Y)\in Cl^0_{odd}(M,E)^2. $ We define, using right invariance, the connection 1-form 
	$$\theta_X Y = \frac{1}{2}\left\{  ad_X Y - ad_\mathbb{A}(X)Y - ad_\mathbb{A}(Y)X\right\}\; .$$
	Then we have that:
	\begin{itemize}
		\item[(a)] \label{p:TF}$\forall (X,Y)\in Cl^0_{odd}(M,E)^2, $ $\theta_X Y - \theta_Y X =  ad_X Y $ 
		(Torsion-free)
			\item[(b)] \label{p:PR} $\forall (X,Y,Z)\in Cl^0_{odd}(M,E)^3, $ 
		$\mathfrak{Re}\left< \theta_X Y ,Z \right>_\mathbb{A} = 
			-\mathfrak{Re}\left<  Y ,\theta_X Z \right>_\mathbb{A} $ (Pseudo-Riemannian metric compatibility) 
	\end{itemize}
Moreover, $\theta: (X,Y) \mapsto \theta_X Y$ is the only bilinear map which satisfies these two properties. 
\end{Theorem}
\begin{proof}
As in the previous section, in this proof we set $\left< \cdot | \cdot \right> = 
\mathfrak{Re}\left< \cdot , \cdot \right>_\mathbb{A}$ for ease of notation.

	We first check that $\theta$ satisfies (a) and (b). 
	By direct computation, we have: 
	\begin{eqnarray*}
\theta_X Y-\theta_YX & = &\frac{1}{2}\left\{ad_XY - ad_\mathbb{A}(X)Y - ad_\mathbb{A}(Y)X\right\} \\
		&& - \frac{1}{2}\left\{ad_YX - ad_\mathbb{A}(Y)X - ad_\mathbb{A}(X)Y\right\} \\
		& = & ad_XY, 
		\end{eqnarray*}
	which proves (a). We now compute using Equation (\ref{ada}): 
	\begin{eqnarray*}
	2\left< \theta_X Y | Z \right> & = & \left< ad_XY - ad_\mathbb{A}(X)Y - ad_\mathbb{A}(Y)X | Z \right> \\
   & = & \left< ad_X Y | Z \right> - \left< ad_\mathbb{A}(X)Y|Z\right> - \left<ad_\mathbb{A}(Y)X | Z \right> \\
	& = & \left< ad_\mathbb{A}(X) Z | Y \right> { -} \left< ad_X Z | Y \right> { - } \left<ad_Y Z | X \right> \\
	& = & \left< ad_\mathbb{A}(X) Z | Y \right> { - } \left< ad_X Z | Y \right> { + } \left<ad_Z Y | X \right> \\
& = &  \left< ad_\mathbb{A}(X) Z | Y \right> - \left< ad_X Z | Y \right> + \left<ad_\mathbb{A}(Z) X | Y \right> \\
		& = & - 2\left<Y | \theta_X Z\right>,
		\end{eqnarray*}
	which proves (b).
	
	Now, let $$\Theta: (X,Y) \in  Cl^0_{odd}(M,E)^2 \mapsto \Theta_XY \in Cl^0_{odd}(M,E)$$ be a bilinear form 
	satisfying (a) and (b).
	Then $$\left< \Theta_X Y | Z \right> + \left<  Y | \Theta_X Z \right> = 0\; ,$$
	$$\left< \Theta_Z X | Y \right> + \left<  X | \Theta_Z Y \right> = 0\; ,$$
	$$\left< \Theta_Y Z | X \right> + \left<  Z | \Theta_Y X \right> = 0\; .$$
	From the third line and (a) we get that 
	$$\left<  Z | \Theta_X Y \right> = -\left< \Theta_Y Z | X \right> - \left<  Z | [X,Y] \right>$$ 
	and from the first line we get that $$\left< \Theta_X Y | Z \right> = - \left<  Y | \Theta_X Z \right>\; .$$
	Combining these two equalities, and exploiting properties (a) and (b), we have: 
	\begin{eqnarray*}
	2 \left< \Theta_X Y | Z \right> & = & 
	               -\left< \Theta_Y Z | X \right> - \left<  Z | [X,Y] \right> - \left<  Y | \Theta_X Z \right> \\
		& = & -\left< -[Y,Z] + \Theta_Z Y | X \right> - \left<  Z | [X,Y] \right> - \left<  Y | -[X,Z] + \Theta_Z X 
		                                                                                               \right> \\
		& = & \left< [Y,Z] | X \right> - \left<  Z | [X,Y] \right> + \left<  Y | [X,Z]  \right> \\
		& = & 2 \left<\theta_XY | Z \right>.
		\end{eqnarray*}
	Since $\left< \cdot\, |\, \cdot \right>$ is non-degenerate, this equality ends the proof.
	\end{proof}

{	
It is important for us to highlight the fact that the proof of Theorem \ref{oneform} goes through because we can use 
the smooth adjoint $ad_\mathbb{A}$. Now, using $\theta_XY$ we can define the curvature operator and sectional 
curvature of the Lie group $Cl^{0,*}_{odd}(M,E)$ as follows:

The curvature operator for the connection $\theta$ is given, at the identity of $Cl^{0,*}_{odd}(M,E)$, by
$$ \text{R}_{\mathbb{A}}(X,Y) = \left[\ \theta_{X}\ ,\ \theta_{Y}\ \right] - \theta_{[X,Y]} $$
for every $X$ and $Y$ in $Cl^{0}_{odd}(M,E)$, see also \cite[Equation (1.10)]{Freed88}.
Hence, the sectional curvature associated to the biplane generated by $X$ and $Y$ is
\begin{equation}
K_{\mathbb{A}}(X,Y) = 
- \frac{\left< \text{R}_{\mathbb{A}}(X,Y)X\ \vert\ Y\right>}{\vert X \wedge Y \vert_{\mathbb{A}}^{2}} \label{delta}
\end{equation}
whenever the area of the parallelogram spanned by $X,Y$, $\vert X \wedge Y \vert_{\mathbb{A}}$, is different from 
zero.

\smallskip

These constructions yield Theorem 5 of Arnold's \cite{fourier}. Using our foregoing notation this theorem reads as 
follows, see \cite[Proposition 2.1]{GPoR}:

\begin{Theorem}  \label{arnold}
Let $\mathbb{A}$ be an inertia operator and set $\mathfrak{N}(X,Y) =
\frac{1}{2}\ ( \text{ad}_{\mathbb{A}}(X)Y + \text{ad}_{\mathbb{A}}(Y) X )$. Given $X$ and $Y$ in $Cl^{0}_{odd}(M,E)$ 
 we have the identity
$$ \vert\ X \wedge Y\ \vert_{A}^{2}\ K_{\mathbb{A}}(X,Y)  =  $$
$$ - \frac{3}{4} \left<\ [X,Y]\, |\, [X,Y]\ \right> 
+ \frac{1}{2} \left<\ [X,Y]\ \vert\ \text{ad}_{\mathbb{A}}(X)Y - \text{ad}_{\mathbb{A}}(Y)X\ \right> $$
$$ + \left<\ \mathfrak{N}(X,Y)\,|\, \mathfrak{N}(X,Y)\ \right>
-\left<\ \mathfrak{N}(X,X)\ \vert\ \mathfrak{N}(Y,Y)\ \right>  \ .  $$
\end{Theorem}

We remark once again that this theorem is a rigorous statement on the sectional curvature of the Fr\'echet Lie
group $Cl^{0,*}_{odd}(M,E)$, not a formal result as \cite[Proposition 2.1]{GPoR}. We finish this section computing 
geodesics:

\smallskip

Let us set $G = Cl^{0,*}_{odd}(M,E)$ and $Lie(G) = Cl^{0}_{odd}(M,E)$.
We recall that a {\em spray} over $G$ is a vector field $S : TG \rightarrow TTG$ satisfying 
$T \pi_{G} \circ S = Id_{TG}\,$, in which $\pi_{G} : TG \rightarrow G$ is the canonical projection, and the
homogeneity condition 
$$
(T\mu_t)\cdot S(v) = \frac{1}{t} S(t v)
$$
for $t \neq 0$, in which $\mu_t : TG \rightarrow TG$ is the smooth function $\mu_t(v) = t\,v$. In the present case 
we use $TG = G \times Lie(G)$ and $TTG = G \times Lie(G) \times Lie(G) \times Lie(G)$. Then
$$
S : T G \rightarrow TTG\, , \quad \quad S(g , X) = (g,X,X, ad_\mathbb{A} (X)X ) = (g,X,X, \theta_X X )
$$
for $g \in G$ and $X \in Lie(G)$, is a spray on $G$, as it can be easily checked (see \cite[Section 1.21]{Mic}).
In actual fact, it can be proven that the spray $S$ is precisely the metric spray corresponding to our 
right-invariant metric (\ref{metric1}), see \cite[Section 6.2]{Ci}. 
The integral curves of $S$ are the {\em geodesics} corresponding to the spray $S$. We obtain that $(g(t),X(t))$ 
is a geodesic for the spray $S$ if and only if
\begin{eqnarray*}
\frac{d g}{dt} & = & X \\
\frac{d X}{d t} & = & ad_\mathbb{A} (X)X \; .
\end{eqnarray*}
The second equation is exactly the Euler-Lagrange equation (\ref{lax2}). 
}

%\begin{remark}
%	We can now check directly that $$\theta_X X = ad_\mathbb{A} (X)X$$ 
%\end{remark}

\section{Examples on the n-dimensional torus}
In this section, we specialize to $M = \mathbb{T}_n = (S^1)^n$, $n$ odd, equipped with its product metric, where 
$S^1 = \R / 2\pi \Z.$
We recall that the Laplace operator is $$\Delta = - \sum_{i = 1}^n \frac{\partial^2}{\partial x_i^2}\; .$$ 

\subsection{When $Q_0$ is a heat operator}
We set $\mathbb{A}=Id.$  
Let $s \in \R_+^*.$ We define $Q_0 = e^{-s\Delta}.$ This is an injective smoothing operator, hence it is both odd 
and even class. We apply our previous computations to $\mathbb{T}_n$ for any $n \in \N^*$ and we obtain
$$\left< A,B \right> = tr^\Delta\left(A e^{-s\Delta}B\right) =  tr\left(A e^{-s\Delta}B\right)\; .$$

Now we need to define formally  $Q_0^{-1} = e^{s\Delta}$ which is not a pseudodifferential operator but it can be
rigorously defined, as we discussed in Remark \ref{nonlocal}.
 
We obtain the formulas:  
 
% $$\frac{\delta H_k}{\delta \mu} = k e^{s\Delta}(P^*)^{k-1},$$
 
 $$ ad_\mathbb{A}(X)Z = [Ze^{-s\Delta},X^*]e^{s\Delta}$$
 and 
$$\theta_XY = \frac{1}{2}\left\{[X,Y] { +} [Ye^{-s\Delta},X^*]e^{s\Delta} { +} 
[Xe^{-s\Delta},Y^*]e^{s\Delta}\right\}\; ,$$
and the geodesic equation (\ref{lax2}) reads 
 $$ \frac{dX}{dt} = [Xe^{-s\Delta},X^*]e^{s\Delta}\; .$$
 Let us now test these three equations taking $X \in C^\infty(\mathbb{T}_n, \C)$ and expanding it with respect to 
 the Fourier basis. First, for $n=1, $ i.e. for $ \mathbb{T}_n = S^1.$ Let $(l,m,p) \in \Z^3.$ We obtain
 \begin{eqnarray*}
 	\left(ad_A(z^l)z^m\right)z^p & = & z^me^{-s\Delta}z^{-l}e^{s\Delta} z^p { - z^{-l+m+p} } \\
 	& = & \left(e^{s\left(p^2 - (p-l)^2\right)}-1\right) z^{-l+m+p}\\
 	& = & \left(e^{s\left(2pl - l^2\right)}-1\right) z^{-l+m+p} \; ,
 	\end{eqnarray*}
and
 \begin{eqnarray*}
 	\left(\theta_{z^l}z^m\right)z^p & = & { \frac{-1}{2} } \left\{ \left(-e^{s\left(2pl - l^2\right)} +1\right) 
 	z^{-l+m+p} +\left(-e^{s\left(2pm - m^2\right)} +1\right) z^{l-m+p} \right\}\; .
 \end{eqnarray*}
% and
% \begin{eqnarray*}
% \frac{d z^l}{dt} = \left( e^{s(2 p l - l^2)}-1\right) z ^p \; .
% \end{eqnarray*}

The same kind of relations can be implemented for $n>1,$ by considering tensor products.

\subsection{When $Q_0$ is a power of the Laplacian}
We set $\mathbb{A}=Id.$  
We now investigate $Q_0 = (\Delta + \pi)^{(n+1)/2},$ where $\pi$ is the $L^2-$ orthogonal projection on the kernel 
of the Laplacian. We remark that $Q_0$ is injective and self-adjoint of order $n+1.$ Moreover, if $n$ is odd, then 
$Q_0$ is odd class, and if $n$ is even, then $Q_0$ is even class.
In this class of examples, we get an operator $Q_0^{-1}$ which is a pseudo-differential operator of order $-n-1,$ in 
the same class as $Q_0.$ Hence the following formulas are fully valid in $Cl_{odd}^0(\mathbb{T}_n,\C): $
 $$ ad_A(X)Z = [ZQ_0,X^*]Q_0^{-1} \in Cl^{-1}_{odd}(\mathbb{T}_n,\C)$$
 
$$ \theta_XY = \frac{1}{2}\left\{[X,Y] - [YQ_0,X^*]Q_0^{-1} -[XQ_0,Y^*]Q_0^{-1}\right\} \; .$$

and the geodesic equation reads 
$$ \frac{dX}{dt} = [XQ_0,X^*]Q_0^{-1} = XQ_0X^*Q_0^{-1} - X X^* \; ,$$
where the right-hand side is an operator of order $-1$ (and hence compact).
Let $M=S^1$ and let us restrict ourselves to $X \in C^\infty(S^1,\C).$
Then, for $p \in \mathbb{Z}$ we obtain 
\begin{equation} \label{int}
\frac{dX}{dt}(z^p) = X(\Delta+\pi)X^*(\Delta+\pi)^{-1}(z^p) - { X X^* } z^p \; .
\end{equation}

Interestingly, in this case we can say more about the integrals of motion $I_k$ considered in Subsection 4.3.
We obtain
$$
I_k = tr^\Delta ( [X (\Delta + \pi) + \xi J^2]^k) = \sum_{j=0}^k \left( \begin{array}{c} h \\ j \end{array} \right) 
tr^\Delta \left( [ X (\Delta + \pi)]^{k-j}J^{2j} \right) \xi^j \; .
$$
Some integrals are trivial (for example, if $j=k$, we obtain the integral of motion $tr^\Delta(J^{2k})$ which does
not depend on the variable $X$) but non-vanishing integrals constructed with different $k$'s cannot 
be all dependent, since the symbols of the pseudodifferential operators $[ X (\Delta + \pi)]^{k-j}J^{2j}$ are all of 
different order. { Thus, it follows that non-vanishing integrals  
$tr^\Delta([ X (\Delta + \pi)]^{k-j}J^{2j})$ are also independent functions.
We can easily prove that indeed there exists a countable family of such non-vanishing functions, at least for a 
large family of initial conditions:

We take $J = Id$ and we evaluate $tr^\Delta([ X (\Delta + \pi)]^{k-j}J^{2j})$ on the initial condition 
$X =  Id + M_{a(x)}e^{-\Delta}M_{\overline{a(x)}}$ in which $M_{a(x)}$ is the multiplication operator by the complex 
valued function $a \in C^\infty(S^1,\C).$ For the sake of simplicity, we take $a(x)=e^{inx}$ for $n \in \Z.$  We 
then have 
\begin{eqnarray*}
	tr^\Delta([ X (\Delta + \pi)]^{k-j}J^{2j}) & = & tr^\Delta((Id +  M_{a(x)}e^{-\Delta}M_{\overline{a(x)}})^{k-j}  
	(\Delta + \pi)^{k-j})\\ & = & tr^\Delta(  (\Delta + \pi)^{k-j}) \\
	&&+ tr(\hbox{trace class, self adjoint positive operator})
\end{eqnarray*}
The second term is strictly positive, while the first term can be computed as the limit of 
$$ 1+ \sum_{k \in \Z^*} (k^2)^{k-j-s}$$ which is equal to $2\zeta(-2k+2j)+1 = \frac{2B_{2(k-j)}}{2k-2j+1}+1$ 
following computations of \cite{CDMP, Ma2020-4} ---in which the authors expand the renormalized trace $tr^\Delta$ in 
the Fourier basis $(x \mapsto e^{inx})_{n \in \Z}\,$--- and the well-known formulas for the $\zeta-$function in 
terms of the Bernoulli numbers \cite{Co}. 

The rigid body equation (\ref{int}) is therefore an example of a non-commutative nonlinear 
differential equation admitting a Lax pair formulation and an infinite number of independent (at least for a large 
number of initial conditions) integrals of motion. It is an integrable equation posed on the Lie algebra of the 
regular Lie group $Cl^{0,*}_{odd}(S^1,E)$.
 }

\medskip

\paragraph{\bf Acknowledgements:} E.G.R.'s research is partially supported by the FONDECYT grant  \#1201894.

\end{document}